\def\thetitle{Toolbox: Gaussian comparison on Eucledian balls}

\def\theruntitle{gaussian comparison}

\def\theabstract{
	In the work a characterization of difference of multivariate Gaussian measures is found on the family of centered Eucledian balls. In particular, it helps to derive 
	\[
	\sup_{t}\left|\P\left(\|\xv_1\|^2<t\right)-\P\left(\|\xv_0\|^2<t\right)\right|\leq \sqrt{\left|Tr\left(I-\left(\Sigma_0\Sigma_1^{-1}+\Sigma_0^{-1}\Sigma_1\right)/2\right)\right|}
	\]
	with vectors $\xv_1\sim\mathcal{N}\left(0,\Sigma_{1}\right)$ and $\xv_0\sim\mathcal{N}\left(0,\Sigma_{0}\right)$.
}

\def\amsPrime{60F17}
\def\amsSec{60F25,60B12}

\def\thekeywords{multivariate Gaussian measure, Kolmogorov distance, Gaussian comparison}

\def\authora{Andzhey Koziuk}
\def\authorb{Vladimir Spokoiny}
\def\runauthora{koziuk, a.}
\def\runauthorb{spokoiny, v.}
\def\addressa{
	Weierstrass Institute, \\
	Institute for Information Transmission Problems of RAS \\
	Mohrenstrasse 39 \\
	10117 Berlin, Germany
}
\def\addressb{
	Weierstrass Institute and \\ Humboldt University Berlin, \\ Moscow Institute of
	Physics and Technology \\
	Mohrenstr. 39, \\
	10117 Berlin, Germany
}
\def\emaila{andzhey.koziuk@wias-berlin.de}
\def\emailb{spokoiny@wias-berlin.de}

\def\month{August}
\def\year{2017}

\newcommand{\ifpaper}[2]{#2}	  
\newcommand{\ifau}[3]{#2}			
\newcommand{\style}[2]{#2}			

\documentclass[11pt]{article}

\usepackage{amsmath,amssymb,amsthm}
\usepackage{epsfig,graphicx}
\usepackage{comment}
\usepackage{color}
\usepackage{srcltx}
\usepackage[mathscr]{eucal}
\usepackage[math]{easyeqn}
\usepackage{etoolbox}
\usepackage[nottoc]{tocbibind}
\usepackage{pdfsync}
\usepackage{mwe}

\def\eqdef{\stackrel{\operatorname{def}}{=}}

\newcommand{\bb}[1]{\boldsymbol{#1}}

\def\Gamma{\varGamma}
\def\Pi{\varPi}
\def\Sigma{\varSigma}
\def\Delta{\varDelta}
\def\Lambda{\varLambda}
\def\Psi{\varPsi}
\def\Phi{\varPhi}
\def\Theta{\varTheta}
\def\Omega{\varOmega}
\def\Xi{\varXi}
\def\Upsilon{\varUpsilon}

\def\R{I\!\!R}
\def\E{I\!\!E}
\def\P{I\!\!P}

\def\kappa{\varkappa}

\def\xv{\bb{x}}
\def\yv{\bb{y}}

\def\Ind{\operatorname{1}\hspace{-4.3pt}\operatorname{I}}

\newcounter{example}[section]
\newcounter{remark}[section]
\numberwithin{equation}{section}
\numberwithin{figure}{section}
\numberwithin{example}{section}
\numberwithin{remark}{section}
\newtheorem{theorem}{Theorem}[section]

\newtheorem{lemma}[theorem]{Lemma}
\newtheorem{corollary}[theorem]{Corollary}

\newtheorem{exmp}[example]{Example}
\newtheorem{rmrk}[remark]{Remark}
\newenvironment{example}{\begin{exmp}\rm}{\end{exmp}}
\newenvironment{remark}{\begin{rmrk}\rm}{\end{rmrk}}

\ifpaper
{
	\title{\thetitle}
	\date{\month, \year \vspace{5cm}}

	\renewenvironment{abstract}
	{\centerline{\textbf{Abstract}}\bigskip
		\begin{center}
			\begin{minipage}{10cm}
				\begin{small}
				}
				{   \end{small}
			\end{minipage}
		\end{center}
		\bigskip
	}

}{
	\title{\thetitle}
	\date{}

	\renewenvironment{abstract}
	{\centerline{\textbf{Abstract}}
		\begin{center}
			\begin{minipage}{15cm}
				\begin{small}
				}
				{   \end{small}
			\end{minipage}
		\end{center}
	}
	
}

\style
{
	\textheight=22cm
	\textwidth=17cm
	\topmargin=0pt
	\oddsidemargin=0.2cm
	\evensidemargin=1cm
	\pagestyle{empty}
	\linespread{1.2}
}{ 
	\textheight=21cm
	\textwidth=17cm
	\topmargin=0pt
	\oddsidemargin=0.2cm
	\evensidemargin=1cm
	\linespread{1.1}
}

\begin{document}
	\thispagestyle{empty}	
	\ifpaper
	{
		\title{\thetitle}
		\ifau{ 
			\author
			{
				\authora
				\ifdef{\thanksa}{\thanks{\thanksa}}{}
				\\[5.pt]
				\addressa \\
				\texttt{ \emaila}
			}
		}
		{  
			\author{
				\authora
				\ifdef{\thanksa}{\thanks{\thanksa}}{}
				\\[5.pt]
				\addressa \\
				\texttt{ \emaila}
				\and
				\authorb
				\ifdef{\thanksb}{\thanks{\thanksb}}{}
				\\[5.pt]
				\addressb \\
				\texttt{ \emailb}
			}
		}
		{   
			\author{
				\authora
				\ifdef{\thanksa}{\thanks{\thanksa}}{}
				\\[5.pt]
				\addressa \\
				\texttt{ \emaila}
				\and
				\authorb
				\ifdef{\thanksb}{\thanks{\thanksb}}{}
				\\[5.pt]
				\addressb \\
				\texttt{ \emailb}
				\and
				\authorc
				\ifdef{\thanksc}{\thanks{\thanksc}}{}
				\\[5.pt]
				\addressc \\
				\texttt{ \emailc}
			}
		}
		\thispagestyle{empty}
		\maketitle
		\pagestyle{myheadings}
		\markboth
		{\hfill \textsc{ \small \theruntitle} \hfill}
		{\hfill
			\textsc{ \small
				\ifau{\runauthora}
				{\runauthora and \runauthorb}
				{\runauthora, \runauthorb, and \runauthorc}
			}
			\hfill}
		\thispagestyle{empty}
		\newpage
		\begin{abstract}
			\theabstract
		\end{abstract}
		\pagenumbering{arabic}

		\par\noindent\emph{AMS 2000 Subject Classification:} Primary \amsPrime. Secondary \amsSec

	}{			
		\ifau{ 
			\author
			{
				\authora
				\ifdef{\thanksa}{\thanks{\thanksa}}{}
				\\[5.pt]
				\addressa \\
				\texttt{ \emaila}
			}
		}
		{  
			\author{
				\authora
				\ifdef{\thanksa}{\thanks{\thanksa}}{}
				\\[5.pt]
				\addressa \\
				\texttt{ \emaila} \\
				\and \\
				\authorb
				\ifdef{\thanksb}{\thanks{\thanksb}}{}
				\\[5.pt]
				\addressb \\
				\texttt{ \emailb}
			}
		}
		{   
			\author{
				\authora
				\ifdef{\thanksa}{\thanks{\thanksa}}{}
				\\[5.pt]
				\addressa \\
				\texttt{ \emaila}
				\and
				\authorb
				\ifdef{\thanksb}{\thanks{\thanksb}}{}
				\\[5.pt]
				\addressb \\
				\texttt{ \emailb}
				\and
				\authorc
				\ifdef{\thanksc}{\thanks{\thanksc}}{}
				\\[5.pt]
				\addressc \\
				\texttt{ \emailc}
			}
		}
		\maketitle
		\begin{abstract}
			\theabstract
		\end{abstract}
		\pagenumbering{arabic}
	}
\section{Introduction}
The work is organized as a report and exposition is kept concise - no proof is deferred, introduction reduced to minimum and conclusion contains problems to address in view of the development. 

The work centers on comparison of the measures. The prerequisite for a study was an idea that difference of multivariate probabilities being more regular than a probability enjoys an independent structure. In essence, regularity gained through subtraction was channeled into a probability to illuminate possibly existing construct. For an explorer it is inevitable to assume a guideline before the factual verification. Luckily enough such a structure was found to exist on a family of multivariate Gaussian measures on centered Eucledian balls.

\section{Associated development}
The most tightly related investigation and motivational examples can be found in {G{\"o}tze}, F. and {Naumov}, A. and {Spokoiny}, V. and {Ulyanov}, V. \cite{2017arXiv170808663G}. Compared to the work the current development is not based on the pdf estimation and derives an explicit characterization for the difference of multivariate Gaussian measures (see corollary [$\ref{L1}$]) .
 
\section{Gaussian comparison on Euclidean balls}\label{GC}
Kolmogorov distance on a class of centered Euclidean balls 
\[
\mathcal{B}_t\eqdef\{\xv\in\R^p\;:\;\|\xv\|<\sqrt{t}\}
\] 
is studied
\[
\sup_{t}\left|\P(\xv_1\in\mathcal{B}_t)-\P(\xv_0\in\mathcal{B}_t)\right|
\]
For independent $\xv_0\sim\mathcal{N}\left(0,\Sigma_0\right)$ and $\xv_1\sim\mathcal{N}\left(0,\Sigma_1\right)$
define a composite vector 
\[
\xv_s=\sqrt{1-s}\xv_0+\sqrt{s}\xv_1
\]
with $s=[0,1]$ and following 
\[
\xv_s\sim\mathcal{N}\left(0,\Sigma_s\right)
\]
where 
\[
\Sigma_s=(1-s)\Sigma_{0}+s\Sigma_{1}.
\]
The vector enables construction of a bridge over the regular difference
\begin{equation}\label{E4}
g_{\alpha}(t)\eqdef\E f\left(\alpha\|\xv_1\|^2-\alpha t\right)-\E f\left(\alpha\|\xv_0\|^2- \alpha t\right)
\end{equation}
\[
\eqdef\int_0^1\E\frac{\partial f_{\alpha}(\|\xv_s\|^2-t)}{\partial s}ds=\alpha\int_0^1\E f_{\alpha}^{\prime}\left(\|\xv_s\|^2- t\right)\xv_s^T\left(\xv_1\frac{1}{\sqrt{s}}-\xv_0\frac{1}{\sqrt{1-s}}\right)ds
\]
using fundamental theorem of calculus and assuming that $f$ is sufficiently smooth.

The most direct tool to work with Gaussian measures is Stein's identity. In the current work one uses an adaptation of the fact.
\begin{lemma}[Stein's lemma]\label{Stein}
	If $\xv\in\R^p$ is a centered Gaussian with a covariance $\Sigma\succ0$ and a vector function $h(\xv):\R^p\rightarrow \R^p$ is an almost differentiable with  $\E\|\frac{\partial h_j(\xv)}{\partial\xv}\|\leq\infty$ then
	\begin{equation}
	\E\xv h^T(\xv)=\Sigma \E \frac{\partial h^T(\xv)}{\partial\xv}.
	\end{equation}
\end{lemma}
\begin{proof}
	In his work Charles M. Stein 1981 \cite{stein1981estimation} (Lemma 2) derived for standard normal $\yv\sim\mathcal{N}(0,I)$ and a function $h_j(\Sigma^{1/2}\yv):\R^p\rightarrow\R$
	\[
	\E\yv h_j(\Sigma^{1/2}\yv)=\E \frac{\partial h_j(\Sigma^{1/2}\yv)}{\partial\yv}=\Sigma^{1/2}\E \frac{\partial h_j(\Sigma^{1/2}\yv)}{\partial\Sigma^{1/2}\yv}.
	\]
	A change of variables $\xv=\Sigma^{1/2}\yv$ yields $\E\xv h_j(\xv)=\Sigma\E \frac{\partial h_j(\xv)}{\partial\xv}$. Therefore, the rules of a matrix multiplication dictate
	\[
	\E\xv h^T(\xv)=\left(\E\xv h_1(\xv),\E\xv h_2(\xv),...,\E\xv h_p(\xv)\right)
	\]
	\[
	=\left(\Sigma\E \frac{\partial h_1(\xv)}{\partial\xv},\Sigma\E \frac{\partial h_2(\xv)}{\partial\xv},...,\Sigma\E \frac{\partial h_p(\xv)}{\partial\xv}\right)=\Sigma\E \frac{\partial h^T(\xv)}{\partial\xv}.
	\]
\end{proof}

For an expository reasons only let us also include in the note another - less straightforward, however applicable in a most general case - relaxation method of a probability. Introduce a smooth indicator function
	\begin{equation}\label{smoothingF}
	f_{\alpha}(x)\eqdef\Ind(x>0)-\frac{1}{2}\textbf{sign}(x)e^{-\alpha\left|x\right|}
	\end{equation}
	and a regular and limiting object based on an integral operator $\mathbb{L}_{\xv}(\cdot)$ (e.g. expectation) 
	\[
	\phi_{\alpha}(t)\eqdef\mathbb{L}_{\xv}\left(f_{\alpha}(\|\xv\|^2 - t)\right)\;\;\text{and}\;\;\phi_{\infty}(t)\eqdef\mathbb{L}_{\xv}\left(\Ind\left(\|\xv\|^2>t\right)\right).
	\]
	The relaxation of the indicator - the kernel ($\ref{smoothingF}$) - is characterized structurally in the next lemma. 
	\begin{lemma}\label{LinSmoothing}
	Assume an integral operator $\mathbb{L}_{\xv}\left(\cdot\right)$ s.t. $\phi_{\alpha}(t)$ is
	smooth bounded and with bounded second derivative. Then $\phi_{\alpha}(t)$ satisfies an ODE
	\[
	\phi_{\alpha}(t)=\phi_{\infty}(t)+\phi^{\prime\prime}_{\alpha}(t)/\alpha^2.
	\]
	Moreover, an ordering holds
	\[
	\forall\alpha>0\;\;\sup_t\left|\phi_{\alpha}(t)\right|\leq\sup_t\left|\phi_{\infty}(t)\right|.
	\]
	\end{lemma}
	\begin{proof} Notice that
	\[
	\mathbb{L}_{\xv}\left(f_{\alpha}(\|\xv\|^2- t)-\Ind\left(\|\xv\|^2>t\right)\right)=\frac{\left(\mathbb{L}_{\xv}\left(f_{\alpha}(\|\xv\|^2- t)\right)\right)^{\prime\prime}_t}{\alpha}.
	\]
	Thus the kernel ($\ref{smoothingF}$) admits an ODE representation
	\[
	\mathbb{L}_{\xv}\left( f_{\alpha}(\|\xv\|^2- t)\right)=\mathbb{L}_{\xv}\left(\Ind\left(\|\xv\|^2>t\right)\right)+\frac{1}{\alpha^2}\left(\mathbb{L}_{\xv} \left( f_{\alpha}(\|\xv\|^2- t)\right)\right)^{\prime\prime}_{t}
	\]
	with an inequality
	\[
	\sup_t\left|\mathbb{L}_{\xv}\left( f_{\alpha}(\|\xv\|^2- t)\right)\right|\leq\sup_t\abs{\mathbb{L}_{\xv}\left(\Ind\left(\|\xv\|^2>t\right)\right)}
	\]
	following from the characterization of extreme points - second derivative in maximum is negative and positive in minimum - and boundedness of the second derivative.
	\end{proof}

The regular difference $g_{\alpha}(t)$ $\left(\ref{E4}\right)$ with the kernel ($\ref{smoothingF}$) can be equivalently rewritten by the means of Stein's lemma [$\ref{Stein}$].
\begin{lemma}\label{L0}
	Denote $\yv\sim \mathcal{N}\left(0,I\right)$ and uniform $s\sim\mathcal{U}(0,1)$ and define a linear operator
	\[
	\mathbb{L}_{\yv,s}\eqdef\frac{1}{2}\E\left(\yv^T\log^{\prime}\Sigma_{s}\yv-\log^{\prime}\det\Sigma_{s}\right)\left(\cdot\right)
	\]
	then the function ($\ref{E4}$) follows
	\[
	g_{\alpha}(t)=\mathbb{L}_{\yv,s}\left(f_{\alpha}\left(\yv^T\Sigma_s\yv - t\right)\right).
	\]
\end{lemma}
\begin{proof}
	Independence of $\xv_0$ and $\xv_1$ and the chain rule for differentiation demonstrate 
	\[
	\frac{1}{\sqrt{1-s}}Tr\E_0\xv_0\E_1\xv^T_sf_{\alpha}^{\prime}\left(\|\xv_s\|^2- t\right)\overset{\textbf{Stein}\;\xv_0}{=}Tr\Sigma_0\E_0\frac{1}{\sqrt{1-s}}\frac{\partial}{\partial \xv_0}\E_1\xv^T_sf_{\alpha}^{\prime}\left(\|\xv_s\|^2- t\right)
	\]
	\[
	=Tr\Sigma_0\E\frac{\partial}{\partial \xv_0}\frac{\partial\xv_0}{\partial\xv_s}\xv^T_sf_{\alpha}^{\prime}\left(\|\xv_s\|^2- t\right)\overset{\textbf{chain}}{=}Tr\Sigma_0\E\frac{\partial}{\partial \xv_s}\xv^T_sf_{\alpha}^{\prime}\left(\|\xv_s\|^2- t\right)
	\]
	\[
	\overset{\textbf{Stein}\;\xv_s}{=}Tr\left(\Sigma_{0}\Sigma_s^{-1}\E\xv_s\xv^T_sf_{\alpha}^{\prime}\left(\|\xv_s\|^2- t\right)\right).
	\]
	The difference can be smoothed further applying again Stein's identity and differentiation by parts
	\[
	Tr\Sigma_{0}\Sigma_s^{-1}\E\left(\xv_s\xv^T_sf_{\alpha}^{\prime}\left(\|\xv_s\|^2- t\right)\right)
	\]
	\[
	=\frac{1}{2\alpha}Tr\Sigma_{0}\Sigma_s^{-1}\E\left(\frac{\partial}{\partial\xv_s}\xv^T_sf_{\alpha}\left(\|\xv_s\|^2- t\right)-f_{\alpha}\left(\|\xv_s\|^2- t\right)\right)
	\]
	\[
	\overset{\textbf{Stein}\;\xv_s}{=}\frac{1}{2\alpha}Tr\Sigma_{0}\Sigma_s^{-1}\E\left(\Sigma_s^{-1}\xv_s\xv^T_sf_{\alpha}\left(\|\xv_s\|^2- t\right)-f_{\alpha}\left(\|\xv_s\|^2-t\right)\right)
	\]
	\[
	=\frac{1}{2\alpha}Tr\Sigma_{0}\Sigma_s^{-1}\E\left(\Sigma_s^{-1}\xv_s\xv^T_s-I\right)f_{\alpha}\left(\|\xv_s\|^2- t\right).
	\]
	The same applied to $\xv_1$ yields analogously
	\[ 
	\frac{1}{\sqrt{s}}Tr\E\xv_1\xv^T_sf_{\alpha}^{\prime}\left(\|\xv_s\|^2- t\right)=\frac{1}{2\alpha}Tr\Sigma_{1}\Sigma_s^{-1}\E\left(\Sigma_s^{-1}\xv_s\xv^T_s-I\right)f_{\alpha}\left(\|\xv_s\|^2-t\right).
	\]
	Define  $\log^{\prime}\Sigma_{s}\eqdef\left(\Sigma_1-\Sigma_0\right)\Sigma_s^{-1}$
 	and note that the distribution of a random vector $\yv\eqdef\Sigma^{-1/2}_s\xv_s$ is standard normal then the change of variables concludes the alternative representation for $g_{\alpha}(t)$ ($\ref{E4}$).
\end{proof}
Moreover, expanding on the lemma it is possible to conclude the characterizing for the difference of multivariate Gaussian measures corollary.
\begin{corollary}\label{L1}
	Assume independent and centered Gaussian vectors $\xv_0,\xv_1\in\R^p$. Denote $\yv\sim \mathcal{N}\left(0,I\right)$ and uniform $s\sim\mathcal{U}(0,1)$ and define a linear operator
	\[
	\mathbb{L}_{\yv,s}\eqdef\frac{1}{2}\E\left(\yv^T\log^{\prime}\Sigma_{s}\yv-\log^{\prime}\det\Sigma_{s}\right)\left(\cdot\right),
	\] 
	then 
	\[
	\P\left(\xv_1\in\mathcal{B}_t\right)-\P\left(\xv_0\in\mathcal{B}_t\right)=\mathbb{L}_{\yv,s}\left(\Ind\left(\yv^T\Sigma_{s}\yv < t\right)\right).
	\]
\end{corollary}
\begin{proof} Notice that
	\[
	\mathbb{L}_{\yv,s}\left(\Ind\left(\yv^T\Sigma_{s}\yv>t\right)\right)=\mathbb{L}_{\yv,s}\left(\Ind\left(\yv^T\Sigma_{s}\yv<t\right)\right),
	\]
	then Lebesgue dominated convergence theorem and the lemma [$\ref{L0}$] enable the statement in the limit $\alpha\rightarrow\infty$. 
\end{proof}
Informally, one can think of differentiation and integration by parts (Stein's identity) applied to non-smooth functions - indicator and Dirac delta function - which gives a hope to built tight or even optimal bound for Gaussian comparison (see theorem [$\ref{GComp}$]). 

The corollary helps to shift analysis from hardly accessible difference to a more constructive object. 
\begin{theorem}[Gaussian comparison]\label{GComp}
	Assume independent and centered Gaussian vectors $\xv_0,\xv_1\in\R^{p}$ with covariance operators $\Sigma_0,\Sigma_{1}\succ0$ respectively. Define uniformly upper bounded constant
	\[
	C_p=\sqrt{\P\left(\yv^TA_{s}\yv<p\right)}<1
	\]
	for $\yv\sim\mathcal{N}(0,I)$, uniformly distributed $s\sim \mathcal{U}\left(0,1\right)$ and a matrix $A_s=\frac{p\log^{\prime}\Sigma_s}{\log^{\prime}\det\Sigma_s}$.
	Then it holds
	\[
	\sup_{t}\left|\P\left(\xv_1\in\mathcal{B}_t\right)-\P\left(\xv_0\in\mathcal{B}_t\right)\right|\leq C_p\sqrt{\left|Tr\left(I-\left(\Sigma_0\Sigma_1^{-1}+\Sigma_0^{-1}\Sigma_1\right)/2\right)\right|}.
	\]
\end{theorem}
\begin{proof} 
	The corollary [$\ref{L1}$] implies
	\[
	\sup_{t}\left|\P\left(\xv_1\in\mathcal{B}_t\right)-\P\left(\xv_0\in\mathcal{B}_t\right)\right|\leq\left|\frac{1}{2}\int_0^1\E\left(\yv^T\log^{\prime}\Sigma_{s}\yv-\log^{\prime}\det\Sigma_{s}\right)\Ind\left(\yv^T\log^{\prime}\Sigma_{s}\yv<\log^{\prime}\det\Sigma_{s}\right)ds\right|
	\]
	then Cauchy-Schwartz inequality gives
	\[
	\sup_{t}\left|\P\left(\xv_1\in\mathcal{B}_t\right)-\P\left(\xv_0\in\mathcal{B}_t\right)\right|
	\]
	\[
	\leq\frac{1}{2}\sqrt{\P\left(\yv^T\log^{\prime}\Sigma_{s}\yv/\log^{\prime}\det\Sigma_{s}<1\right)}\sqrt{\int_0^1\E\left(\yv^T\log^{\prime}\Sigma_{s}\yv-\log^{\prime}\det\Sigma_{s}\right)^2ds}.
	\]
	Observe additionally that $\forall s\in\left(0,1\right)$ - $Tr\left(\log^{\prime}\Sigma_s\right)^2=-Tr\left(\log^{\prime\prime}\Sigma_s\right)$, then cross terms cancel out and one attains
	\[
	\sqrt{2\int_0^1\E\left(\yv^T\log^{\prime}\Sigma_{s}\yv-\log^{\prime}\det\Sigma_{s}\right)^2ds}=\sqrt{2\int_0^1Tr\left(\log^{\prime}\Sigma_s\right)^2ds}
	\]
	\[
	=\sqrt{2\left|Tr\int_0^1\left(\log\Sigma_s\right)^{\prime\prime}ds\right|}=\sqrt{2\left|Tr\left(\log^{\prime}\Sigma_s\Big|_{s=1}-\log^{\prime}\Sigma_s\Big|_{s=0}\right)\right|}
	\]
	\[
	=2\sqrt{\left|Tr\left(I-\left(\Sigma_0\Sigma_1^{-1}+\Sigma_0^{-1}\Sigma_1\right)/2\right)\right|}.
	\]
 	The statement readily follows uniformly in dimension.
\end{proof}
\begin{rmrk}
	Notice that the theorem has no condition $p>2$ typical for a density based approximation of Kolmogorov distance.
\end{rmrk}
\begin{rmrk}
	Also observe that $Tr A_s=p$ and for a fixed $s$ if there exists $\delta\neq0$ such that $\left|A_s-I\right|\succ \delta I$ then the limit $\lim_{p\rightarrow\infty}p^k\P\left(\yv^TA_{s}\yv<p\right)=0$ tends to zero as there is always a gap $\delta_s>0$ s.t. $\P\left(\yv^TA_{s}\yv<p\right)=\P\left(\|\yv\|^2<p-\delta_s\right)$.
\end{rmrk}
A useful implication of the result [$\ref{L1}$] is the anti-concentration bound. With an a-priori chosen shift $\Delta>0$ it is equivalent to bounding Kolmogorov distance with scaled covariances
\begin{equation}\label{S}
	\P\left(\xv\in\mathcal{B}_{t+\Delta}\right)-\P\left(\xv\in\mathcal{B}_{t}\right)
\end{equation}
\[
\leq\sup_{t^{\prime}}\left|\P\left(\sqrt{\left(1+\Delta/t\right)}\xv\in\mathcal{B}_{t^{\prime}}\right)-\P\left(\xv\in\mathcal{B}_{t^{\prime}}\right)\right|=\sup_{t}\left|\P\left(\xv_0\in\mathcal{B}_{t}\right)-\P\left(\xv_1\in\mathcal{B}_{t}\right)\right|
\]
with $\Sigma_0=\Sigma$ and$\Sigma_1=\left(1+\Delta/t\right)\Sigma$. Even though the theorem [$\ref{GComp}$] straightforwardly applies let us justify the anti-concentration independently to cross-validate the comparison result.
\begin{theorem}[Anti-concentration]\label{AC1}
	For a centered Gaussian vector $\xv\sim\mathcal{N}\left(0,\Sigma\right)$ in $\R^p$
	\[
	\P\left(\xv\in\mathcal{B}_{t+\Delta}\right)-\P\left(\xv\in\mathcal{B}_{t}\right)\leq Cp^{1/2}\left(\log\left(1+\Delta/t\right)\wedge\sqrt{\log\left(1+\Delta/t\right)}\right)
	\]
	with a universal constant $C<\infty$.
\end{theorem}
\begin{proof}
	Choose $\beta>1$, a covariance $\Sigma\succ 0$ and define 
	\[
	\xv_0=\xv,\;\;\xv_1=\beta^{1/2}\xv,\;\;\beta_{s}=(1-s)+s\beta\;\;\text{and}\;\;\Sigma_{s}=(1-s)\Sigma+s\beta\Sigma.
	\]
	The corollary [$\ref{L1}$] claims that
	\[
	\sup_{t}\left|\P\left(\xv_1\in\mathcal{B}_t\right)-\P\left(\xv_0\in\mathcal{B}_t\right)\right|=\sup_t\left|\frac{1}{2}Tr\int_0^1\E\left(\yv\yv^T-I\right)\left(\log\Sigma_{s}\right)^{\prime}\Ind\left(\yv^T\Sigma_{s}\yv < t\right)ds\right|,
	\]
	which in the particular case is written as
	\[
	\sup_{t}\left|\P\left(\beta^{1/2}\xv\in\mathcal{B}_t\right)-\P\left(\xv\in\mathcal{B}_t\right)\right|=\sup_t\left|\frac{1}{2}\int_0^1\left(\log\beta_{s}\right)^{\prime}\E\left(\|\yv\|^2-p\right)\Ind\left(\yv^T\Sigma\yv < t/\beta_{s}\right)ds\right|.
	\]
	One had to exploit the logarithmic structure $\left(\log\Sigma_{s}\right)^{\prime}=\left(\log\beta_s\right)^{\prime}$ to attain the expression. Therefore, it holds
	\[
	\sup_{t}\left|\P\left(\beta^{1/2}\xv\in\mathcal{B}_t\right)-\P\left(\xv\in\mathcal{B}_t\right)\right|\leq\frac{1}{2}\E\left|\|\yv\|^2-p\right|\left|\int_0^1\left(\log\beta_{s}\right)^{\prime}ds\right|
	\]
	\[
	\leq\frac{1}{2}\E\left|\|\yv\|^2-p\right|\log\beta\leq C p^{1/2}\left(\log\beta\wedge\sqrt{\log\beta}\right)
	\]
	with a universal constant
	\[
	\forall p\;:\;C>\E\left|\frac{\|\yv\|^2-p}{2\sqrt{p}}\right|\vee\E\left|\frac{\|\yv\|^2\log\beta-p\log\beta}{2\sqrt{p\log\beta}}\right|.
	\]
	To complete the derivation it suffice to insert $\beta\eqdef1+\Delta/t$ and use the observation ($\ref{S}$) in front of the corollary.
\end{proof}
There exists an alternative technique reproducing derived anti-concentration. It provides a convenient instrument to validate the characterization [$\ref{L1}$].
\begin{theorem}\label{PAC}
	For a Gaussian vector $\xv\in\R^p$ with a covariance $\Sigma\succ 0$ it holds
	\[
	\P\left(\xv\in\mathcal{B}_{t+\Delta}\right)-\P\left(\xv\in\mathcal{B}_{t}\right)\leq\sqrt{p}\Delta/t.
	\]
\end{theorem}    
\begin{proof}
	Pinsker's inequality states
	\[
	\P\left(\xv\in\mathcal{B}_{t+\Delta}\right)-\P\left(\xv\in\mathcal{B}_{t}\right)\leq\sqrt{\mathcal{K}(\P_1,\P_2)/2}.
	\]
	Explicitly, the Kullback-Leibler divergence between $\P_1=\mathcal{N}\left(0,\Sigma\right)$ and $\P_2=\mathcal{N}\left(0,\frac{t^2}{(t+\Delta)^2}\Sigma\right)$ is written
	\[
	\mathcal{K}(\P_1,\P_2)=p/2((\Delta/t)^2+2(\Delta/t)-2\log(1-\Delta/t))\leq p(\Delta/t)^2.
	\]
	The combination of the two concludes $\P\left(\xv\in\mathcal{B}_{t+\Delta}\right)-\P\left(\xv\in\mathcal{B}_{t}\right)\leq \sqrt{p}\Delta/t$.
\end{proof}
The result approximately matches the corollary [$\ref{AC1}$] and an exact relation is explained via an inequality
\[
\P\left(\xv\in\mathcal{B}_{t+\Delta}\right)-\P\left(\xv\in\mathcal{B}_{t}\right)\overset{corollary \ref{AC1}}{<}C\sqrt{p}\log\left(1+\Delta/t\right)<C\sqrt{p}\Delta/t.
\]
Up to a constant factor the corollary [$\ref{AC1}$] is sharper than the theorem [$\ref{PAC}$]. However, they match in order and no significant improvement is gained compared to Pinsker's inequality.

Interestingly the theorem [$\ref{AC1}$] can be used to bound a pdf $\rho_{\|\xv\|^2}(t)\eqdef\E\delta\left(\|\xv\|^2-t\right)$ of squared Eucledian norm of a Gaussian vector 
\[
\rho_{\|\xv\|^2}(t)=\lim_{\Delta\rightarrow 0^+}\frac{\P\left(\xv\in\mathcal{B}_{t+\Delta}\right)-\P\left(\xv\in\mathcal{B}_{t}\right)}{\Delta}\leq\lim_{\Delta\rightarrow 0^+}\frac{Cp^{1/2}\log\left(1+\Delta/t\right)}{\Delta}=C\frac{\sqrt{p}}{t}.
\]  
\begin{corollary}\label{D}
	A Gaussian vector $\xv\sim\mathcal{N}\left(0,\Sigma\right)$ follows 
	\[
	\rho_{\|\xv\|^2}(t)\leq C\sqrt{p}/t
	\] 
	with a universal constant $C<\infty$.
\end{corollary}
Moreover, the corollary [$\ref{L1}$] additionally enables a structural characterization of the density.
 \begin{theorem}\label{CHD}
 	A Gaussian vector $\xv\sim\mathcal{N}\left(0,\Sigma\right)$ follows 
 	\[
 	\rho_{\|\xv\|^2}(t)=\frac{1}{2t}\E_{\yv^T\Sigma\yv<t}\left(p-\|\yv\|^2\right)
 	\] 
 	with the standard normal $\yv\sim\mathcal{N}\left(0,I\right)$.
 \end{theorem}
\begin{proof}
	Once again fix $\beta>1$, a covariance $\Sigma\succ 0$ and define 
	\[
	\xv_0=\xv,\;\;\xv_1=\beta^{1/2}\xv,\;\;\beta_{s}=(1-s)+s\beta\;\;\text{and}\;\;\Sigma_{s}=(1-s)\Sigma+s\beta\Sigma.
	\]
	The corollary [$\ref{L1}$] claims that
	\[
	\P\left(\xv_1\in\mathcal{B}_{t}\right)-\P\left(\xv_0\in\mathcal{B}_t\right)=\frac{1}{2}Tr\int_0^1\E\left(\yv\yv^T-I\right)\left(\log\Sigma_{s}\right)^{\prime}\Ind\left(\yv^T\Sigma_{s}\yv < t\right)ds,
	\]
	which is written here as
	\[
	\P\left(\xv\in\mathcal{B}_{t+\Delta}\right)-\P\left(\xv\in\mathcal{B}_t\right)=\P\left(\left(1+\Delta/t\right)^{-1/2}\xv\in\mathcal{B}_{t}\right)-\P\left(\xv\in\mathcal{B}_{t}\right)
	\]
	\[
	=\frac{1}{2}\int_0^1\left(\log\beta_{s}\right)^{\prime}\E\left(\|\yv\|^2-p\right)\Ind\left(\beta_{s}\yv^T\Sigma\yv < t\right)ds.
	\]
	choosing $\beta\eqdef\left(1+\Delta/t\right)^{-1}$ and using the logarithmic structure $\left(\log\Sigma_{s}\right)^{\prime}=\left(\log\beta_s\right)^{\prime}$. 
	
	Therefore, taking the limit $\Delta\rightarrow 0$ one finds a description of the probability density function
	\[
	\rho_{\|\xv\|^2}(t)=\lim_{\Delta\rightarrow 0^+}\frac{\P\left(\xv\in\mathcal{B}_{t+\Delta}\right)-\P\left(\xv\in\mathcal{B}_{t}\right)}{\Delta}
	\]
	\[
	=\lim_{\beta\rightarrow 1^-}\frac{\beta-1}{2t\left(1/\beta-1\right)}\int_0^1\frac{1}{\beta_{s}}\E\left(\|\yv\|^2-p\right)\Ind\left(\beta_{s}\yv^T\Sigma\yv < t\right)ds=\frac{1}{2t}\E\left(p-\|\yv\|^2\right)\Ind\left(\yv^T\Sigma\yv < t\right).
	\]
\end{proof}
However, the state of the art bound on the density can be found in {G{\"o}tze}, F. and {Naumov}, A. and {Spokoiny}, V. and {Ulyanov}, V. \cite{2017arXiv170808663G}.

\section{Conclusion}
The work allows to hope that an advancement is also possible for the density estimation. One can derive a refined version of the corollary [$\ref{D}$].

Another clear application objective is a bootstrap. It corresponds to a covariance operator empirically estimating the other one. The procedure is generally helpful for a practitioner constructing confidence sets.

Hierarchically even deeper challenge to answer is whether there exist and what are the other structurally stable probabilistic phenomena in multivariate or Hilbert spaces. Particularly, aside from Stein's lemma substituted on Stein-Chen's lemma the core of the analysis remains indicating a possible gain in the direction.  
\bibliographystyle{plain}
\bibliography{Article}
\end{document}